\documentclass[12pt,reqno]{amsart}
\usepackage{amssymb,amsmath,amsthm}
\usepackage{comment,enumitem,url}
\usepackage{xcolor}
\usepackage{cleveref}
\usepackage{mathtools}

\newtheorem{theorem}{Theorem}
\newtheorem{lemma}[theorem]{Lemma}
\newtheorem{corollary}[theorem]{Corollary}

\newtheorem*{remark}{Remark}
\newtheorem*{remarks}{Remarks}

\theoremstyle{remark}

\theoremstyle{definition}

\numberwithin{theorem}{section} \numberwithin{equation}{section}

\setlist[enumerate]{leftmargin=*,label=\rm{(\arabic*)}}
\setlist[itemize]{leftmargin=*}

\renewcommand{\Mc}{\mathcal{M}}

\renewcommand{\l}{\lambda}

\newcommand{\R}{\mathbb{R}}

\newcommand{\coeff}{\operatorname{coeff}}

\newcommand{\pa}[2]{\left(\frac{#1}{#2}\right)}

\allowdisplaybreaks

\author{Walter Bridges and Kathrin Bringmann}

\title{Log concavity for unimodal sequences}

\date{\today}



\begin{document}
\maketitle

\begin{abstract}
	In this paper, we prove that the number of unimodal sequences of size $n$ is log-concave. These are coefficients of a mixed false modular form and have a Rademacher-type exact formula due to recent work of the second author and Nazaroglu on false theta functions. Log-concavity and higher Tur\'an inequalities have been well-studied for (restricted) partitions and coefficients of weakly holomorphic modular forms, and analytic proofs generally require precise asymptotic series with error term. In this paper, we proceed from the exact formula for unimodal sequences to carry out this calculation. We expect our method applies to other exact formulas for coefficients of mixed mock/false modular objects.
\end{abstract}

\section{Introduction and statement of results}

In \cite{DP}, DeSalvo and Pak proved that $p(n)$, the number of partitions of $n$, is log-concave for $n\ge26$. That is,
$$
	p(n)^2-p(n-1)p(n+1) > 0, \qquad \text{for $n \geq 26$.}
$$
This follows for large enough $n$ directly from the Hardy--Ramanujan asymptotic expansion for $p(n)$ (see \cite[\S 6.2]{DP}), but log-concavity for $n\ge26$ requires a careful argument. DeSalvo and Pak proceeded from Rademacher's exact formula for $p(n)$, together with work of Lehmer \cite{Lehmer1,Lehmer2}. There has since been a flurry of results studying log-concavity and higher Tur\'an inequalities for partition generating functions and weakly holomorphic modular forms (see for example \cite{CJW,GORZ,OPR}).

In this paper, we consider unimodal sequences, which distinguish themselves from the other examples by their connection to false theta functions. Let $u(n)$ count the number of {\it unimodal sequences} of $n$,
$$
	1\leq a_1\leq \dots \leq a_r \leq c \geq b_s \geq \dots \geq b_1 \geq 1, \quad \sum_{j=1}^r a_j + c + \sum_{j=1}^s b_j=n, \quad r,s \in \mathbb{N}_0.
$$
DeSalvo and Pak \cite{DP} mentioned that asymptotic formulas for unimodal sequences (or stacks/convex compositions) are not precise enough to prove log-concavity. We overcome this problem, using recent work of the second author and Nazaroglu, proving an exact formula for $u(n)$. Here, we work out precise asymptotic expansions with explicit error term.

Throughout, we write $f(n)=O_{\le c}(g(n))$ if $|f(n)|\le c|g(n)|$ and set $n_0:=100\,000$.

\begin{theorem}\label{T:main}
	For $n\ge n_0$, we have
	$$
		u(n) = \frac{e^{2\pi\sqrt\frac n3}}{n^\frac54}\left(A+\frac{B}{\sqrt n}+\frac Cn+\frac{D}{n^\frac32}+\frac{E}{n^2}+O_{\le 478}\left(\frac{1}{n^{\frac{5}{2}}}\right)\right),
	$$
	where the constants $A,B,C,D$, and $E$ are defined in equation \eqref{E:ABCDEdef}. 
\end{theorem}

\begin{remarks}
	\ \begin{enumerate}
		\item Presumably both the constant in the error term as well as the number of terms may be improved significantly with our methods. We only stop at the power $n^{-2}$ to obtain log-concavity.
		
		\item We expect our method applies to other exact formulas for coefficients of mixed mock/false modular objects. Indeed, Mauth in \cite{M} follows this approach to prove log-concavity for so-called partitions without sequences considered by the authors in \cite{BB}. Another example arises from irreducible characters of certain vertex operator algebras considered by Cesana \cite{C}.
	\end{enumerate}
\end{remarks}

The following asymptotic expansion is a direct consequence of \Cref{T:main}.

\begin{corollary}\label{C:LogConAsymp}
	For $n\ge n_0$, we have
	\begin{multline*}
		u(n)^2 - u(n-1)u(n+1)\\
		= \frac{e^{4\pi\sqrt\frac n3}}{n^{\frac52+\frac32}}\left(\frac{\pi A^2}{2\sqrt3} + \left(-\frac{5A^2}{4}+\frac{\pi AB}{\sqrt3}-B^2\right)n^{-\frac12} + O_{\le106}\left(n^{-1}\right)\right).
	\end{multline*}
\end{corollary}

Combined with a numerical check (see the remark in Section 2), log-concavity follows.

\begin{corollary}\label{C:main}
	For $n \in \mathbb{N} \setminus \{1,5,7\}$, we have
	\[
		u(n)^2 - u(n-1)u(n+1) > 0.
	\]
\end{corollary}

\begin{remark}
	Note that a combinatorial proof of the log-concavity for integer partitions is still open. Indeed, as there are so many failures of log-concavity (i.e., for $0\le n\le25$), this is likely a difficult problem. The number of exceptions for unimodal sequences by contrast is much more manageable, so perhaps it is reasonable to ask for a combinatorial proof in this case.
\end{remark}

The paper is organized as follows. In \Cref{S:pre}, we recall the exact formula for $u(n)$ from \cite{BN} and state some inequalities for the $I$-Bessel function. The proof of \Cref{T:main} is carried out in Sections \ref{S:kge2} and \ref{S:main}: In \Cref{S:kge2}, we bound the contribution to $u(n)$ of the terms for $k\ge2$ in \Cref{T:ExactFormula}, in \Cref{S:main}, we use the saddle-point method to prove an asymptotic expansion for the term $k=1$ in \Cref{T:ExactFormula}, finishing the proof of \Cref{T:main}.

\section*{Acknowledgements}

We are grateful to Lukas Mauth for helping us verify log-concavity for the remaining cases $7\le n\le n_0$ with Wolfram Mathematica. The authors have received funding from the European Research Council (ERC) under the European Union's Horizon 2020 research and innovation programme (grant agreement No. 101001179).

\section{Preliminaries}\label{S:pre}

Recall that the generating function for unimodal sequences is given by (see \cite[equation (3.2) and (3.3)]{A})
\begin{equation}\label{E:GenFn}
	\sum_{n \geq 0} u(n)q^n=\sum_{n \geq 0} \frac{q^n}{(q;q)_n^2}=\frac{1}{(q;q)_{\infty}^2}\sum_{n \geq 0} (-1)^nq^{\frac{n(n+1)}{2}}.
\end{equation}

\begin{remark}\label{R:computation}
	Recall that $(q;q)_\infty^{-1}$ is the partition generating function. The right-hand side of \eqref{E:GenFn} allows for a quick computation of $u(n)$ as a convolution of pairs of partitions and the coefficients $\{0,\pm1\}$ in the sparse series. This is especially true with programs like Wolfram Mathematica which have the partitions of large order already hard-coded.
\end{remark}

In \cite{BN}, we found the following exact formula for $u(n)$. Note that if $u^*(n)$ denotes the unimodal sequences counted in \cite{BN}, then $u(n)=\coeff_{[q^n]}\frac{1}{(q;q)_\infty^2}-u^*(n)$ (see \cite[footnote on page 4]{BB}).

\begin{theorem}[\cite{BN}, Theorem 1.3, negative of the second term]\label{T:ExactFormula}
	We have
	\begin{multline*}
		u(n) = \frac{\pi}{2^\frac34\sqrt3(24n+1)^\frac34}\sum_{k\ge1} \sum_{r=0}^{2k-1} \frac{K_k(n,r)}{k^2}\\
		\times \int_{-1}^1 \left(1-x^2\right)^\frac34 \cot\left(\frac{\pi}{2k}\left(\frac{x}{\sqrt6}-r-\frac12\right)\right) I_\frac32\left(\frac{\pi}{3\sqrt2k}\sqrt{\left(1-x^2\right)(24n+1)}\right) dx,
	\end{multline*}
	where $I_\frac32$ is the $I$-Bessel function of order $\frac32$ and $K_k(n,r)$ is a certain Kloostermann-type sum (in particular $|K_k(n,r)|\le k$).
\end{theorem}

We require the following bounds for the $I_\frac32$--Bessel function.

\begin{lemma}\label{lem:first}
	We have
	\[
		I_\frac32(y) \le
		\begin{cases}
			\sqrt\frac{2}{\pi y}e^y & \text{if }y\ge1,\\
			\vspace{-.45cm}\\
			\frac{2\sqrt2}{3\sqrt\pi}y^\frac32 & \text{if }0\le y<1.
		\end{cases}
	\]
\end{lemma}

\section{The terms $k\ge2$}\label{S:kge2}

In this section we bound the contribution from $k\ge2$ in \Cref{T:ExactFormula}. We begin by estimating the sum on $r$.

\begin{lemma}\label{L:rbound}
	For $|x|\leq 1$ and $k \geq 2$, we have
	\[
		\left|\sum_{r=0}^{2k-1} K_k(n,r)\cot\left(\frac{\pi}{2k}\left(\frac{x}{\sqrt6}-r-\frac12\right)\right)\right| \le \frac{4k^2}{\pi}(\log(k)+14).
	\]
\end{lemma}

\begin{proof}
	 We use for $0\leq y \leq \pi$
	\[
		|\cot(y)| \le \frac{1}{\min\{y,\pi-y\}}
	\]
	to bound
	$$
		\left|\cot\left(\frac{\pi}{2k}\left(\frac{x}{\sqrt6}-r-\frac12\right)\right)\right| = \left|\cot\left(\frac{\pi}{2k}\left(-\frac{x}{\sqrt6}+r+\frac12\right)\right)\right|.
	$$
	A short calculation shows that for $|x|\le1$,
	\begin{multline*}
		\min\left\{\frac{\pi}{2k}\left(-\frac{x}{\sqrt6}+r+\frac12\right), \pi-\frac{\pi}{2k}\left(-\frac{x}{\sqrt6}+r+\frac12\right)\right\}\\
		\ge
		\begin{cases}
			\frac{\pi}{2k}(r+0.09) & \text{if }0\le r\le k-1,\\
			\frac{\pi}{2k}(2k-r-0.91) & \text{if }k\le r\le2k-1.
		\end{cases}
	\end{multline*}
	Thus, using $|K_k(n,r)|\le k$, we have
	\begin{align*}
		&\left|\sum_{r=0}^{2k-1} K_k(n,r)\cot\left(\frac{\pi}{2k}\left(\frac{x}{\sqrt6}-r-\frac12\right)\right)\right| \\ &\leq \sum_{r=0}^{k-1} k\left|\cot\left(\frac{\pi}{2k}\left(-\frac{x}{\sqrt6}+r+\frac12\right)\right)\right| + \sum_{r=k}^{2k-1} k\left|\cot\left(\frac{\pi}{2k}\left(-\frac{x}{\sqrt6}+r+\frac12\right)\right)\right| \\ &\le \frac{2k^2}{\pi}\left(\sum_{r=0}^{k-1} \frac{1}{r+0.09}+\sum_{r=k}^{2k-1} \frac{1}{2k-r-0.91}\right) = \frac{4k^2}{\pi}\sum_{r=0}^{k-1} \frac{1}{r+0.09}.
	\end{align*}
	The claim now follows from the integral comparison criterion.
\end{proof}

Now we bound the sum over from all $k\ge2$ in \Cref{T:ExactFormula} as follows.

\begin{lemma}\label{L:kgeq2bound}
	For $n\ge n_0$, we have
	\begin{multline*}
		\Bigg|\frac{\pi}{2^\frac34\sqrt3(24n+1)^\frac34}\sum_{k\ge2} \sum_{r=0}^{2k-1} \frac{K_k(n,r)}{k^2}\\
		\times \int_{-1}^1 \left(1-x^2\right)^\frac34 \cot\left(\frac{\pi}{2k}\left(\frac{x}{\sqrt6}-r-\frac12\right)\right) I_\frac32\left(\frac{\pi}{3\sqrt2k}\sqrt{\left(1-x^2\right)(24n+1)}\right)dx \Bigg|\\
		= O_{\le1}\left(e^{\pi\sqrt\frac n3}\right).
	\end{multline*}
\end{lemma}

\begin{proof}
	By \Cref{L:rbound}, we can bound the left-hand side by
	\begin{multline}\label{E:k2}
		\frac{2^\frac94}{\sqrt3(24n+1)^\frac34}\sum_{k\ge2} (\log(k)+14)\\
		\times \int_0^1 \left(1-x^2\right)^\frac34 I_\frac32\left(\frac{\pi}{3\sqrt2k}\sqrt{\left(1-x^2\right)(24n+1)}\right) dx.
	\end{multline}
	To estimate the integral of the $I_\frac32$-Bessel function, we apply \Cref{lem:first}.
	
	First we consider the range
	\[
		0 \le \frac{\pi}{3\sqrt2k}\sqrt{\left(1-x^2\right)(24n+1)} < 1 \Leftrightarrow x^2 > 1 - \frac{18k^2}{\pi^2(24n+1)}.
	\]
	Applying \Cref{lem:first} and extending the range of integration, we have
	\begin{multline*}
		\int_{\sqrt{1-\frac{18k^2}{\pi^2(24n+1)}}}^1 \left(1-x^2\right)^\frac34 I_\frac32\left(\frac{\pi}{3\sqrt2k}\sqrt{\left(1-x^2\right)(24n+1)}\right) dx\\
		\le \int_0^1 \left(1-x^2\right)^\frac34 \frac{2\sqrt2}{3\sqrt\pi} \left(\frac{\pi}{3\sqrt2k}\sqrt{\left(1-x^2\right)(24n+1)}\right)^\frac32 dx.
	\end{multline*}
	Using the evaluation
	\[
		\int_0^1 \left(1-x^2\right)^\frac32 dx = \frac{3\pi}{16},
	\]
	this part contributes overall
	\begin{align}
		\frac{\pi^2}{18}\sum_{k\ge2} \frac{\log(k)+14}{k^\frac32} &\le 10.3. \label{E:kgeq2largerange}
	\end{align}

	Next we consider the range 
	\[
	 	\frac{\pi}{3\sqrt2k}\sqrt{\left(1-x^2\right)(24n+1)} \geq 1 \Leftrightarrow x^2 \leq 1 - \frac{18k^2}{\pi^2(24n+1)}.
	\]
	If $k\ge\frac{\pi\sqrt{24n+1}}{3\sqrt2}$, then this range is empty and we have no contribution. Otherwise, \Cref{lem:first} gives that the corresponding contribution to \eqref{E:k2} can be bound against
	\begin{align*}
		&\frac{2^\frac94}{\sqrt3(24n+1)^\frac34}\sum_{2\le k<\frac{\pi\sqrt{24n+1}}{3\sqrt2}} (\log(k)+14)\int_0^{\sqrt{1-\frac{18k^2}{\pi^2(24n+1)}}} \left(1-x^2\right)^\frac34\\
		&\hspace{4cm}\times \sqrt\frac2\pi \left(\frac{\pi}{3\sqrt2k}\sqrt{\left(1-x^2\right)(24n+1)}\right)^{-\frac12} e^{\frac{\pi}{3\sqrt2k}\sqrt{\left(1-x^2\right)(24n+1)}} dx\\
		&\hspace{1cm}= \frac{8}{\pi(24n+1)}\sum_{2\le k\le\frac{\pi\sqrt{24n+1}}{3\sqrt2}} \sqrt k(\log(k)+14)\\
		&\hspace{6.3cm}\times \int_0^{\sqrt{1-\frac{18k^2}{\pi^2(24n+1)}}} \left(1-x^2\right)^\frac12 e^{\frac{\pi}{3\sqrt{2}k}\sqrt{(1-x^2)(24n+1)}} dx.
	\end{align*}

	We now trivially bound this against the exponential
	\begin{multline}\label{E:kgeq2smallrange}
		\frac{8}{ \pi(24n+1)}\left(\frac{\pi\sqrt{24n+1}}{3\sqrt2}\right)^\frac32 \left(\log\left(\frac{\pi\sqrt{24n+1}}{3\sqrt2}\right)+14\right) e^\frac{\pi\sqrt{24n+1}}{6\sqrt2}\\
		= O_{\le 0.9}\left(e^{\pi\sqrt\frac n3}\right),
	\end{multline}
	where in the last step we use $n\ge n_0$. Combining \eqref{E:kgeq2largerange} and \eqref{E:kgeq2smallrange} proves the lemma.
\end{proof}

\section{The main term and the proof of \Cref{T:main}}\label{S:main}

The term from $k=1$ in \Cref{T:ExactFormula} equals
\begin{multline}\label{E:k1}
	\frac{2^\frac14\pi}{3^\frac12(24n+1)^\frac34}\int_{-1}^1 \left(1-x^2\right)^\frac34\cot\left(\frac\pi2\left(\frac{x}{\sqrt6}+\frac12\right)\right)\\
	\times I_\frac32\left(\frac{2\pi}{\sqrt3}\sqrt{\left(1-x^2\right)\left(n+\frac{1}{24}\right)}\right).
\end{multline}
Using
\[
	I_\frac32(y) = \frac{1}{\sqrt{2\pi y}}\left(\left(1-\frac1y\right)e^y+\left(1+\frac1y\right)e^{-y}\right),
\]
we obtain that \eqref{E:k1} equals
\begin{multline}\label{E:k1int}
	\frac{1}{24n+1}\int_{-1}^1 \cot\left(\frac\pi2\left(\frac{x}{\sqrt6}+\frac12\right)\right) \left(\left(\sqrt{1-x^2}-\frac{3\sqrt2}{\pi\sqrt{24n+1}}\right) e^{\frac{\pi}{3\sqrt2}\sqrt{\left(1-x^2\right)(24n+1)}}\right.\\
	\left.+ \left(\sqrt{1-x^2}+\frac{3\sqrt2}{\pi\sqrt{24n+1}}\right) e^{-\frac{\pi}{3\sqrt2}\sqrt{\left(1-x^2\right)(24n+1)}}\right) dx.
\end{multline}

We bound the second term in \eqref{E:k1int} for $n\ge n_0$ as
\begin{multline*}
	\left|\frac{1}{24n+1}\int_{-1}^1 \cot\left(\frac\pi2\left(\frac{x}{\sqrt6}+\frac12\right)\right) \left(\sqrt{1-x^2}+\frac{3\sqrt2}{\pi\sqrt{24n+1}}\right) e^{-\frac{\pi}{3\sqrt2}\sqrt{\left(1-x^2\right)(24n+1)}} dx\right|\\
	\hspace{4.4cm}\le \frac{1}{24n_0+1}\int_{-1}^1 7\cdot(1+1)\cdot 1 dx \le 0.1.
\end{multline*}

We split the integral for the first term in \eqref{E:k1int} as
\[
	\int_{|x|\le1} = \int_{|x|\le n^{-\frac18}} + \int_{n^{-\frac18}\le|x|\le1}.
\]
We bound the contribution from the second range as
\begin{align}\nonumber
	&\left|\frac{1}{24n+1}\int_{n^{-\frac18}\le|x|\le1} \cot\left(\frac\pi2\left(\frac{x}{\sqrt6}+\frac12\right)\right) \left(\sqrt{1-x^2}-\frac{3\sqrt2}{\pi\sqrt{24n+1}}\right)\right.\\
	\nonumber
	&\hspace{10.5cm}\times \left.{\vphantom{\frac{\sqrt2}{\sqrt2}}} e^{\frac{\pi}{3\sqrt2}\sqrt{\left(1-x^2\right)(24n+1)}} dx\right|\\
	\label{E:boundawayfromsaddlepoint}
	&\hspace{2.4cm}\le \frac{1}{24n+1}\cdot2\cdot7\cdot1\cdot e^{\frac{\pi}{3\sqrt2}\sqrt{\left(1-n^{-\frac14}\right)(24n+1)}} \le \frac{14}{24n+1}e^{2\pi\sqrt\frac n3-\frac{\pi n^\frac14}{\sqrt3}}. 
\end{align}

The rest of this section is devoted to obtaining an asymptotic expansion for 
\begin{multline}\label{E:near0int}
	\frac{1}{24n+1}\int_{-n^{-\frac18}}^{n^{-\frac18}} \cot\left(\frac\pi2\left(\frac{x}{\sqrt6}+\frac12\right)\right) \left(\sqrt{1-x^2}-\frac{3\sqrt2}{\pi\sqrt{24n+1}}\right)\\
	\times e^{\frac{\pi}{3\sqrt2}\sqrt{\left(1-x^2\right)(24n+1)}} dx
\end{multline}
of the form in Theorem \ref{T:main}. That is, our error term needs to take the shape
\[
	\frac{e^{2\pi\sqrt\frac n3}}{n^\frac54}O\left(n^{-\frac{5}{2}}\right).
\]
In keeping with the saddle-point method, we eventually let $x\mapsto cn^{-\frac14}x$ for some constant $c$, so $dx\mapsto cn^{-\frac14}dx$, and thus we obtain the factor $O(n^{-\frac54})$ outside of the integral. Hence, we need to expand the integrand itself up to $O(n^{-\frac52})$. To that end, we begin with the following lemma. Set
\[
	y = y_n(x) := \frac{\pi}{3\sqrt2}\sqrt{24n+1}\sum_{m=2}^{11} \binom{\frac12}{m}(-1)^mx^{2m}.
\]

\begin{lemma}\label{L:exp}
	For $|x|\le n^{-\frac18}$, we have
	\begin{multline*}
		e^{\frac{\pi}{3\sqrt2}\sqrt{\left(1-x^2\right)(24n+1)}} = e^{\frac{\pi}{3\sqrt2}\sqrt{24n+1}\left(1-\frac{x^2}{2}\right)} \left(1 + y + \frac{y^2}{2} + \frac{y^3}{6} + \frac{y^4}{24} + O_{\le0.34}\left(n^{\frac{5}{2}}x^{20}\right)\right)\\
		\times \left(1+O_{\le0.73}\left(n^{-\frac{5}{2}}\right)\right).
	\end{multline*}
\end{lemma}

\begin{proof}
	By expanding the Taylor series for $\sqrt{1-x^2}$ and noting that $|x|\le n^{-\frac18}\le\frac12$, we have
	\begin{multline*}
		\exp\left(\frac{\pi}{3\sqrt2}\sqrt{\left(1-x^2\right)(24n+1)}\right) \\= \exp\left(\frac{\pi}{3\sqrt2}\sqrt{24n+1}\left(1-\frac{x^2}{2}\right) + y +\frac{\pi}{3\sqrt2}\sqrt{24n+1}O_{\le0.1}\left(x^{24}\right)\right).
	\end{multline*}
	Now using that $e^u\le1+2|u|$ for $u\in[-1,1]$ and $|x|\le n^{-\frac18}$, we have
	\[
		\exp\left(\frac{\pi}{3\sqrt2}\sqrt{24n+1}O_{\le0.1}\left(x^{24}\right)\right) = 1 + O_{\le0.73}\left(n^{-\frac{5}{2}}\right).
	\]
	Next, we note that
	$$
		|y| \leq \frac{\pi}{3\sqrt{2}}\sqrt{25n}x^4 \sum_{m=2}^{11} \left|\binom{\frac{1}{2}}{m}\right| \leq 1.3\sqrt{n}x^4,
	$$
	so that in particular $|y|\le1.3$ for $|x|\le n^{-\frac18}$. Hence,
	\[
		\left|e^y-\sum_{j=0}^4 \frac{y^j}{j!}\right| \le \sum_{j\ge5} \left(\frac{|y|}{2}\right)^j \le \frac{|y|^5}{2^5}\frac{1}{1-0.65} \le 0.34n^\frac52x^{20}.
	\]
	The lemma follows.
\end{proof}

Next, we require the Taylor expansions of the other functions in the integrand in \eqref{E:near0int}, namely 
\begin{multline}\label{E:cotTaylor2}
	\cot\left(\frac\pi2\left(\frac{x}{\sqrt6}+\frac12\right)\right)\\
	= 1 + \frac{\pi^2x^2}{12} + \frac{5\pi^4x^4}{864} + \frac{61\pi^6x^6}{155520} 
	+\frac{277\pi^8x^8}{10450944}+ O_{\le 0.6}\left(x^{10}\right)+P_{\text{odd}}(x),
\end{multline}
where $P_{\text{odd}}(x)$ is an odd polynomial, so does not contribute to the integral, and also
\begin{equation}\label{E:sqrtTaylor}
	\sqrt{1-x^2}=1-\frac{x^2}{2}-\frac{x^4}{8}-\frac{x^6}{16}-\frac{5x^8}{128} + O_{\leq 0.3}\left(x^{10} \right).
\end{equation}
Thus, by \Cref{L:exp} and equations \eqref{E:cotTaylor2} and \eqref{E:sqrtTaylor}, we see that \eqref{E:near0int} equals
\begin{align}\nonumber
	&\frac{e^{\frac{\pi}{3\sqrt2}\sqrt{24n+1}}}{24n+1}\int_{-n^{-\frac18}}^{n^{-\frac18}} \left(1 + \frac{\pi^2x^2}{12} + \frac{5\pi^4x^4}{864} + \frac{61\pi^6x^6}{155520} + \frac{277\pi^8x^8}{10450944} + O_{\le0.6}\left(x^{10}\right)\right)\\
	\nonumber
	&\hspace{1cm}\times \left(1 - \frac{x^2}{2} - \frac{x^4}{8} - \frac{x^6}{16} -\frac{5x^8}{128} + O_{\le0.3}\left(x^{10}\right) - \frac{3\sqrt2}{\pi\sqrt{24n+1}}\right)e^{-\frac{\pi}{3\sqrt2}\sqrt{24n+1}\frac{x^2}{2}}\\
	\nonumber
	&\hspace{2cm}\times \left(1 + y_n(x) + \frac{y_n(x)^2}{2} + \frac{y_n(x)^3}{6} + \frac{y_n(x)^4}{24} + O_{\le0.34}\left(n^\frac52x^{20}\right)\right)\\
	\label{cotint}
	&\hspace{9.2cm}\times \left(1+O_{\le0.73}\left(n^{-\frac52}\right)\right) dx.
\end{align}
Set $\lambda_n:=(\frac{\pi}{6\sqrt{2}}\sqrt{24n+1})^{\frac{1}{2}}$. It is not hard to see that
\begin{equation}\label{E:lambdabounds}
	1.3 n^{\frac{1}{4}}\leq \lambda_n\leq 1.4 n^{\frac{1}{4}}.
\end{equation}
Next we make the change of variables $x\mapsto\frac{x}{\l_n}$ in \eqref{cotint} to get
\begin{align}\nonumber
	&\frac{e^{\frac{\pi}{3\sqrt2}\sqrt{24n+1}}}{(24n+1)\l_n} \int_{-\l_nn^{-\frac18}}^{\l_nn^{-\frac18}} e^{-x^2}\\
	\nonumber
	&\hspace{2cm}\times \left(1 + \frac{\pi^2x^2}{12\l_n^2} + \frac{5\pi^4x^4}{864\l_n^4} + \frac{61\pi^6x^6}{155520\l_n^6} + \frac{277\pi^8x^8}{10450944\l_n^8} + O_{\le0.6}\pa{x^{10}}{\l_n^{10}}\right)\\
	\nonumber
	&\hspace{2cm}\times \left(1 - \frac{x^2}{2\l_n^2} - \frac{x^4}{8\l_n^4} - \frac{x^6}{16\l_n^6}-\frac{5x^8}{128\l_n^8} + O_{\le0.3}\pa{x^{10}}{\l_n^{10}} - \frac{1}{2\l_n^2}\right)\\
	\nonumber
	&\hspace{2cm}\times \left(1 + y_n\left(\frac{x}{\l_n}\right) + \frac{y_n\pa{x}{\l_n}^2}{2} + \frac{y_n\pa{x}{\l_n}^3}{6} + \frac{y_n\pa{x}{\l_n}^4}{24} + O_{\le0.03}\pa{x^{20}}{\l_n^{10}}\right)\\
	\label{E:IntWithTaylor}
	&\hspace{2cm}\times \left(1+O_{\le0.73}\left(n^{-\frac52}\right)\right) dx.
\end{align}
Now note that
\begin{align}
	y_n\left(\frac{x}{\lambda_n}\right)& = -\frac{x^4}{4\lambda_n^2}-\frac{x^6}{8\lambda_n^4}-\frac{5x^8}{64\lambda_n^6}-\frac{7x^{10}}{128\lambda_n^8}+2x^2F\left(\frac{x}{\lambda_n}\right), \label{E:yintermsofF}
\end{align}
where for $|X|\le0.5$
\[
	F(X) := \sum_{m=5}^{10} \binom{\frac12}{m+1}(-1)^mX^{2m} = O_{\le0.03}\left(X^{10}\right).
\]
Thus $2F(\frac{x}{\l_n})=O_{\le0.06}(\frac{x^{10}}{\l_n^{10}})$.

We now plug \eqref{E:yintermsofF} into \eqref{E:IntWithTaylor} and simplify (using a computer algebra system as aide) to get 
\begin{align}\nonumber
	&\frac{e^{\frac{\pi}{3\sqrt2}\sqrt{24n+1}}}{(24n+1)\l_n}\int_{-\l_nn^{-\frac18}}^{\l_nn^{-\frac18}} e^{-x^2}\\
	\nonumber
	&\times \left(1+\left(-\frac12+\frac{\pi^2-6}{12}x^2-\frac{x^4}{4}\right)\l_n^{-2}+\left(-\frac{\pi^2}{24}x^2+\frac{5\pi^4-36\pi^2}{864}x^4-\frac{\pi^2}{48}x^6+\frac{x^8}{32}\right)\l_n^{-4}\right.\\
	\nonumber
	&\hspace{1cm}+ \left(-\frac{5\pi^4}{1728}x^4+\frac{61\pi^6-450\pi^4}{155520}x^6-\frac{5\pi^4}{3456}x^8+\frac{405\pi^2+2430}{155520}x^{10}-\frac{x^{12}}{384}\right)\l_n^{-6}\\
	\nonumber
	&\hspace{1.5cm}+ \left(-\frac{61\pi^6}{311040}x^6+\left(\frac{277\pi^8}{10450944}-\frac{61\pi^6}{311040}\right)x^8-\frac{61\pi^6}{622080}x^{10}\right.\\
	\nonumber
	&\hspace{2cm}\left.\left.+ \left(\frac{5\pi^4}{27648}+\frac{\pi^2}{768}+\frac{7}{768}\right)x^{12} + \left(-\frac{\pi^2}{4608}-\frac{1}{384}\right)x^{14}+\frac{1}{6144}x^{16}\right)\l_n^{-8}\right.\\
	\label{E:IntFullTaylor}
	&\hspace{10.6cm}\left.+E_n\left(x\right){\vphantom{\frac{x^4\pi^4}{48}}}\right) dx,
\end{align}
where $E_n(x)$ contains only powers $\lambda_n^{-m}$ with $m \geq 10$. We bound the integral of $E_n$ explicitly using a computer algebra system to carry out the integral and arrive at
\begin{equation}\label{E:Enbound}
	\int_{-\l_nn^{-\frac18}}^{\l_nn^{-\frac18}} e^{-x^2}\left|E_n(x)\right| dx \le \int_{-\infty}^\infty e^{-x^2}\left|E_n(x)\right|dx \le 5362n^{-\frac52}.
\end{equation}

Now we use that for $m$ even with $0\le m\le16$,
\[
	\int_{-\infty}^{\infty} x^me^{-x^2}dx = \frac{(m-1)!!}{2^{\frac{m}{2}}} \sqrt{\pi}.
\]
We now use for $w\in\R_{\ge1}$ and $m$ even with $0\le m\le16$,
$$
	\left|\frac{(m-1)!!}{2^{\frac{m}{2}}} \sqrt{\pi}-\int_{-w}^{w} x^me^{-x^2}dx\right| \leq 2.8w^me^{-w^2} \frac{(m-1)!!}{2^{\frac{m}{2}}} \sqrt{\pi}.
$$
Define
$$
	\Mc_{k,m}:=\sup\left\{2.8\frac{1.4^m}{1.3^k}\frac{(m-1)!!}{2^\frac m2}\sqrt\pi n^{\frac m8-\frac k4+\frac{5}{2}}e^{-1.69n^\frac14} : n\geq n_0\right\}.
$$ 
With \eqref{E:lambdabounds}, we have
\begin{equation}\label{E:Mkmbound}
	\l_n^{-k}\int_{-\l_nn^{-\frac18}}^{\l_nn^{-\frac18}} x^me^{-x^2} dx = \frac{(m-1)!!\sqrt\pi}{2^\frac m2\l_n^k} + O_{\le\Mc_{k,m}}\left(n^{-\frac{5}{2}}\right).
\end{equation}
We use \eqref{E:Mkmbound} in \eqref{E:IntFullTaylor}, simplify and combine with \eqref{E:Enbound}. Thus \eqref{E:IntFullTaylor} equals
\begin{align}\nonumber
	&\frac{e^{\frac{\pi}{3\sqrt2}\sqrt{24n+1}}}{(24n+1)\l_n}\left(\sqrt\pi+\frac{\left(5308416\pi^2-119439360\right)\sqrt\pi}{127401984}\l_n^{-2}\right.\\
	\nonumber
	&\hspace{2cm}+ \frac{\left(552960\pi^4-11612160\pi^2+26127360\right)\sqrt\pi}{127401984}\l_n^{-4}\\
	\nonumber
	&\hspace{2cm}+ \frac{\left(93696\pi^6-2177280\pi^4+9797760\pi^2+4898880\right)\sqrt\pi}{127401984}\l_n^{-6}\\
	\nonumber
	&\hspace{2cm}+ \frac{\left(22160\pi^8-579744\pi^6+3742200\pi^4-2245320\pi^2+2525985\right)\sqrt\pi}{127401984}\l_n^{-8}\\
	\label{E:mainwithlambda}
	&\hspace{2cm}\left.{\vphantom{\frac{\left(\pi^2\right)\sqrt\pi}{7}}}+ O_{\le5427}\left(n^{-\frac52}\right)\right).
\end{align}

Finally, we apply to \eqref{E:mainwithlambda} the Taylor expansions for $n\ge n_0$
\begin{align*}
	e^{\frac{\pi}{3\sqrt2}\sqrt{24n+1}} &= e^{2\pi\sqrt\frac n3}\left(1 + \frac{\pi}{24\sqrt3\sqrt n} + \frac{\pi^2}{3456n} + \frac{34560\pi^3-3732480\pi}{8599633920\sqrt3n^\frac32}\right.\\
	&\hspace{3.5cm}\left.\times \left(\frac{\pi^4}{71663616} - \frac{\pi^2}{165888}\right)n^{-2} + O_{\le0.001}\left(n^{-\frac52}\right)\right),\\
	\frac{1}{(24n+1)\l_n} &= \frac{1}{8\cdot3^\frac34\sqrt\pi n^\frac54} \left(1-\frac{5}{96n}+\frac{5}{2048n^2}+O_{\le0.00001}\left(n^{-\frac52}\right)\right)\\
	\l_n^{-2} &= \frac{\sqrt3}{\pi n^\frac12} - \frac{1}{16\sqrt3\pi n^\frac32} + O_{\le0.001}\left(n^{-\frac52}\right),\\
	\l_n^{-4} &= \frac{3}{\pi^2n} - \frac{1}{8\pi^2n^2} + O_{\le0.00001}\left(n^{-\frac52}\right),\qquad \l_n^{-6} = \frac{3\sqrt3}{\pi^3n^\frac32} + O_{\le0.02}\left(n^{-\frac52}\right),\\
	\l_n^{-8} &= \frac{9}{\pi^4n^2} + O_{\le0.001}\left(n^{-\frac52}\right).
\end{align*}

Thus \eqref{E:mainwithlambda} equals
\begin{align}\nonumber
	&\frac{e^{2\pi\sqrt\frac n3}}{8\cdot3^\frac34\sqrt\pi n^\frac54}\left(\sqrt\pi+\left(\frac{\pi^\frac32}{6\sqrt3}-\frac{15\sqrt3}{16\sqrt\pi}\right)n^{-\frac12}+\left(\frac{13\pi^\frac52}{864}-\frac{35\sqrt\pi}{96}+\frac{315}{512\pi^\frac32}\right)n^{-1}\right.\\
	\nonumber
	&\hspace{1.5cm}+ \left(\frac{7\pi^\frac72}{972\sqrt3}-\frac{91\pi^\frac32}{512\sqrt3}+\frac{315\sqrt3}{1024\sqrt\pi}+\frac{945\sqrt3}{8192\pi^\frac52}\right)n^{-\frac32}\\
	\nonumber
	&\hspace{3cm}+ \left(\frac{7441\pi^\frac92}{4478976}-\frac{77\pi^\frac52}{1728}+\frac{5005\sqrt\pi}{16384}-\frac{3465}{16384\pi^\frac32}+\frac{93555}{524288\pi^\frac72}\right)n^{-2}\\
	\label{E:saddleptfinal}
	&\hspace{10.1cm}\left.{\vphantom{\frac{\pi^\frac52}{\pi^\frac32}}}+ O_{\le5429}\left(n^{-\frac52}\right)\right).
\end{align}
 Define
\begin{align}
	A&:=\frac{1}{8\cdot3^{\frac{3}{4}}},\nonumber \\
	B&:={\frac{\pi}{144\cdot3^\frac14}}-{\frac {5\cdot{3}^{\frac{3}{4}}}{128\pi}},\nonumber \\
	C&:={\frac {105\cdot{3}^\frac14}{4096{\pi}^{2}}}+{\frac {13\pi^{2}}{6912\cdot3^\frac34}}-{\frac {35}{768\cdot3^\frac34}},\nonumber \\
	D&:=-{\frac {91\pi}{12288\cdot3^\frac14}}+{\frac {105 \cdot {3}^{\frac34}}{8192\pi}}+{\frac {7\pi^{3}}{23328\cdot3^\frac14}}+{\frac {315\cdot{3}^{\frac34}}{65536{\pi}^{3}}},\nonumber \\
	E&:={\frac {7441\pi^{4}}{35831803\cdot3^\frac34}}-{\frac {77\pi^{2}}{13824\cdot3^\frac34}}+{\frac {5005}{131072\cdot3^\frac34}}-{
	\frac {1155\cdot{3}^\frac14}{131072{\pi}^{2}}}+{\frac {31185\cdot{3}^\frac14}{4194304{\pi}^{4}}}. \label{E:ABCDEdef}
\end{align}
We now combine \eqref{E:saddleptfinal} with \Cref{L:kgeq2bound} and \eqref{E:boundawayfromsaddlepoint}.
From
\begin{align*}
	O_{\le1}\left(e^{\pi\sqrt\frac n3}\right) &= O_{\le0.5}\left(\frac{e^{2\pi\sqrt\frac n3}}{8\cdot3^\frac34\sqrt\pi n^\frac54}n^{-\frac52}\right),\\
	\frac{14}{24n+1}e^{2\pi\sqrt\frac n3-\frac{\pi n^\frac14}{\sqrt3}} &= O_{\le10000}\left(\frac{e^{2\pi\sqrt\frac n3}}{8\cdot 3^\frac34\sqrt\pi n^\frac54}n^{-\frac52}\right),
\end{align*}
we conclude \Cref{T:main}.

\end{document}